\documentclass[paper=a4paper,namedreferences]{kluwer}
\usepackage{stmaryrd}
\usepackage{proof}
\usepackage{QED}

\newcommand{\Bm}{\mbox{\BID}^{-}}
\newcommand{\BID}{\textrm{BID}}
\newcommand{\Pow}{\mathcal{P}}
\newcommand{\PP}{\mathbf{P}}
\newcommand{\eqW}{\simeq_{W}}

\newcommand{\Vaan}{V\"a\"an\"anen}
\newcommand{\Var}{\mathcal{V}}
\newcommand{\lrW}{\lrarr_{W}}
\newcommand{\modX}{\models_{X}}
\newcommand{\nmodX}{\not\models_{X}}
\newcommand{\MM}{\mathcal{M}}
\newcommand{\LL}{\mathcal{L}}
\newtheorem{theorem}{Theorem}

\newtheorem{proposition}[theorem]{Proposition}

\newcommand{\vsn}{\vspace{-.1in}}

\newcommand{\somepi}{\exists(\pi)}
\newcommand{\allpi}{\forall(\pi)}
\newcommand{\someH}{\exists_{H}}
\newcommand{\allH}{\forall_{H}}
\newcommand{\someW}{\exists_{W}}
\newcommand{\allW}{\forall_{W}}
\newcommand{\UU}{\mathsf{U}}
\newcommand{\VVV}{\mathsf{V}}
\newcommand{\WW}{\mathsf{W}}
\newcommand{\dset}{{\downarrow}}
\newcommand{\lsem}{\llbracket}
\newcommand{\rsem}{\rrbracket}

\newcommand{\lrarr}{\longrightarrow}
\newcommand{\IFF}{\;\; \Longleftrightarrow \;\;}
\newcommand{\CC}{\mathcal{C}}

\newcommand{\Real}{\mathbb{R}}

\newcommand{\id}[1]{\mathsf{id}_{#1}}

\newcommand{\op}[1]{#1^{\mathsf{op}}}
\newcommand{\CCop}{\op{\CC}}

\newcommand{\Pos}{\mathbf{Pos}}

\newcommand{\ie}{\textit{i.e.}\ }

\newcommand{\pow}[1]{\mathcal{P}(#1)}
\newcommand{\rarr}{\rightarrow}

\newcommand{\linimpl}{\multimap}

\renewcommand{\emph}[1]{\textbf{#1}}

\newcommand{\vn}{\varnothing}

\newcommand{\HH}{\mathcal{H}}

\newcommand{\bit}{\begin{itemize}}
\newcommand{\eit}{\end{itemize}\par\noindent}
\newcommand{\ben}{\begin{enumerate}}
\newcommand{\een}{\end{enumerate}\par\noindent}
\newcommand{\beq}{\begin{equation}}
\newcommand{\eeq}{\end{equation}\par\noindent}
\newcommand{\beqa}{\begin{eqnarray*}}
\newcommand{\eeqa}{\end{eqnarray*}\par\noindent}
\newcommand{\beqn}{\begin{eqnarray}}
\newcommand{\eeqn}{\end{eqnarray}\par\noindent}

\begin{document}
\begin{article}
\begin{opening}

\title{From IF to BI}
\subtitle{A Tale of Dependence and Separation}

\author{Samson \surname{Abramsky}}
\institute{Oxford University Computing Laboratory}
\author{Jouko \surname{V\"a\"an\"anen}}
\institute{ILLC Amsterdam}

\begin{abstract}
We take a fresh look at the logics of informational dependence and independence of Hintikka
 and Sandu and \Vaan, and their compositional semantics due to Hodges. We show how Hodges' semantics can be seen as a special case of a general construction, which provides a context for a useful completeness theorem with respect to a wider class of models. We shed some new light on each aspect of the logic. We show that  the natural propositional logic carried by the semantics is the logic of Bunched Implications due to Pym and O'Hearn, which combines intuitionistic and multiplicative connectives. This introduces several new connectives not previously considered in logics of informational dependence, but which we show play a very natural r\^ole, most notably intuitionistic implication. As regards the quantifiers, we show that their interpretation in the Hodges semantics is forced, in that they are the image under the general construction of the usual Tarski semantics; this implies that they are adjoints to substitution, and hence uniquely determined. As for the dependence predicate, we show that this is definable from a simpler predicate, of constancy or dependence on nothing. This makes essential use of the intuitionistic implication. The Armstrong axioms for functional dependence are then recovered as a standard set of axioms for intuitionistic implication. We also prove a full abstraction result in the style of Hodges, in which the intuitionistic implication plays a very natural r\^ole.
\end{abstract}

\end{opening}

\section{Introduction}
Our aim in this paper is to take a fresh look at the logics of informational dependence and independence \cite{HS89,HS96,Vaan}, and their compositional semantics due to Wilfrid Hodges \cite{Hod97a,Hod97b}. We shall focus on Dependence Logic, introduced by the second author \cite{Vaan}.

The main objective of Hodges' work was to provide a compositional model-theoretic semantics for the IF-logic of Hintikka and Sandu \cite{HS89,HS96}, which matched their ``game-theoretical semantics''. This was achieved by lifting the standard Tarski semantics of first-order formulas, given  in terms of satisfaction in a structure with respect to an assignment to the free variables, to satisfaction by \emph{sets of  assignments}.

We seek a deeper understanding of Hodges' construction:
\begin{itemize}
\item First and foremost, what is going on? Where does the Hodges construction come from? Is it canonical in any way? Why does it work? What structures are really at play here?

\item Because of the equivalence of Dependence Logic (or variants such as IF-logic) under this semantics to  (a significant fragment of) second-order logic, there is no hope for a completeness theorem. But we may get a useful completeness theorem with respect to a wider class of models. Understanding the general algebraic context for the semantics points the way to such a completeness notion.

\item We can  also look for \emph{representation theorems}, with some infinitary ingredients.
\end{itemize}
The results of our investigation are quite surprising conceptually (at least to us). The main points can be summarized as follows.
\begin{itemize}
\item We find a general context for Hodges' construction. We shall not treat it in full generality here, as the general account is best stated in the language of categorical logic \cite{Law69,Pitts}, and we wish to avoid undue technicalities. However, we will indicate the possibilities for a general algebraic semantics, as the basis for a useful completeness theorem.
\item We find that the natural propositional logic associated with the Hodges construction is the \emph{logic of Bunched Implication} of Pym and O'Hearn \cite{OHP99,Pym02}, which combines intuitionistic and multiplicative linear connectives. 
\item This not only yields a more natural view of the strangely asymmetric notions of conjunction and disjunction in the Hodges semantics (one is intuitionistic, while ``disjunction'' is actually multiplicative \emph{conjunction}!),  it also brings into prominence some connectives not previously considered in the setting of IF-logic or Dependence logic, in particular \emph{intuitionistic implication}. This enables a novel analysis of the Dependence predicate of \cite{Vaan}, as a Horn clause with respect to a more primitive predicate of single-valuedness. The well-known Armstrong axioms for functional dependence \cite{Arm} then fall out as a standard axiomatization of intuitionistic (but not classical!) implication.
\item Intuitionistic implication also plays a natural r\^ole in our version of a full abstraction theorem in the sense of Hodges.
\item The construction is shown to lift the interpretation of the standard quantifiers in a canonical way, so that \emph{quantifiers are uniquely determined as the adjoints to substitution} \cite{Law69}, just as in the standard Tarski semantics of first-order logic. This is also extended to characterizations of the dependence-friendly quantifiers of \cite{Vaan} as adjoints.
\end{itemize}

\noindent The plan of the remainder of the paper is as follows. In the next section we provide background on branching quantifiers, IF-logic, dependence logic, and Hodges'  semantics. Then  in section~3 we show how the Hodges semantics is an instance of a general algebraic construction, in which the connectives of BI-logic arise naturally. In section~4, we show that the interpretation of the quantifiers in the Hodges construction is the canonical lift of the standard  interpretation of the quantifiers as adjoints, and hence is uniquely determined.  We also use the intuitionistic implication to show how the dependence-friendly quantifiers can be interpreted as certain adjoints. In section~5, we  show how the intuitionistic implication arises naturally in the proof of a full abstraction theorem. In section~6, we show how the dependence predicate can be analyzed in terms of a more primitive predicate of single-valuedness, using the intuitionistic implication. This turns the ``Armstrong axioms'' into standard theorems of intuitionistic implicational logic. The final section outlines some further directions.

\section{Dependence, Independence and Information Flow}
We begin with a standard example: the formal definition of continuity for a function $f : \Real \lrarr \Real$ on the real numbers.
\[ \forall x. \, \forall \epsilon. \, \exists \delta . \, \forall x'. \, | x - x'| < \delta \; \Rightarrow \; | f(x) - f(x') | < \epsilon \, . \]
This definition is often explained in current calculus courses in terms of an ``epsilon-delta game''.\endnote{See e.g. online resources such as\\ \texttt{http://library.wolfram.com/infocenter/MathSource/4734/}.} The Adversary proposes a number, $\epsilon$, as a measure of how close we must stay to the value of $f(x)$; we must then respond with a number, $\delta$, such that, whenever the input is within the interval $(x - \delta, x+\delta)$, the output does indeed pass the $\epsilon$-test of closeness to $f(x)$. Clearly, the choice of $\delta$ will depend on that of $\epsilon$; the nesting of the quantifiers expresses this dependency.

This is the definition of \emph{global continuity} of $f$, expressed in terms of \emph{local continuity} at every point $x$. This means that the choice of $\delta$ will depend, not only on $\epsilon$, but on $x$ also. Now consider the definition of \emph{uniform continuity}:
\[ \forall \epsilon. \, \exists \delta . \, \forall x. \, \forall x'. \, | x - x'| < \delta \; \Rightarrow \; | f(x) - f(x') | < \epsilon \, . \]
Here $\delta$ still depends on $\epsilon$, \emph{but must be chosen independently of $x$}. This variation in dependency is tracked syntactically by the different order of the quantifiers.  Indeed,  it seems that it was only after the distinction between pointwise and uniform notions of continuity, and, especially, convergence, had been clarified in 19th-century analysis, that the ground was prepared for the introduction of predicate calculus.

More generally, dependence or independence of bounds on various parameters is an important issue in many results on estimates in number theory and analysis.
Hodges quotes a nice example from one of Lang's books \cite{Lan64} in \cite{Hod97a}.

Intuitively, there is an evident relation between these notions and that of \emph{information flow}. Dependence indicates a form of  information flow; independence is  the \emph{absence} of information flow.

\subsection{Beyond first-order logic}
It turns out that mere rearrangement of the order of quantifiers in first-order formulas is not sufficient to capture the full range of possibilities for informational dependence and independence. This was first realized almost 50~years ago, with Henkin's introduction of  \emph{branching quantifiers} \cite{Hen61}. The simplest case is the eponymous \emph{Henkin quantifier}:
\[ \left( \begin{array}{cc} \forall x & \exists y \\
\forall u & \exists v
\end{array}   \right) A(x, y, u, v) .
\]
The intention is that $y$ must be chosen depending on $x$, but \emph{independently} of the choice of $u$; while $v$ must be chosen depending on $u$, but \emph{independently} of the choice of $x$.
The meaning of this formula can be explicated by introducing \emph{Skolem functions} $f$ and $g$: an equivalent formula will be
\[ \exists f. \, \exists g. \, \forall x. \, \forall u. \, A(x, f(x), u, g(u)) . \]
Here the constraints on dependencies are tracked by the dependence of the Skolem functions on certain variables, but not on others.
Note that the Skolemized sentence is \emph{second-order}; in fact, it belongs to the $\Sigma_{1}^{1}$ fragment of second-order logic.\endnote{This can be described as the fragment comprising formulas $\exists f_{1} \ldots \exists f_{n}. \, \phi$,  where the $f_{i}$ are function variables, and $\phi$ is a first-order formula over a signature extended by these function variables.} This second-order rendition of the meaning of the Henkin quantifier cannot be avoided, in the sense that  the Henkin quantifier strictly increases the expressive power of first-order logic, and in fact the extension of first-order logic with the Henkin quantifier is equivalent in expressive power to  the $\Sigma_{1}^{1}$ fragment \cite{Hen61}.

\paragraph{Examples}
\begin{enumerate}
\item Consider
\[ \left( \begin{array}{cc} \forall x & \exists y \\
\forall u & \exists v
\end{array} \right)  (A(x) \rarr B(y)) \; \wedge \; (B(u) \rarr A(v)) \; \wedge \; [(x = v) \; \leftrightarrow \;  (y = u)].
\]
This expresses that $A$ and $B$ are equinumerous sets.
\item Now consider
\[   \begin{array}{ll}
\exists v. \, \left( \begin{array}{cc} \forall x_{1} & \exists y_{1} \\
\forall x_{2} & \exists y_{2}
\end{array} \right) (A(x_{1}) \rarr A(y_{1}))   & \wedge \; [(x_{2} = y_{1}) \; \rightarrow \;  (y_{2} = x_{1})] \\    & \wedge \; A(v)  \; \wedge \;  (A(x_{1}) \rightarrow (y_{1} \neq v)) \, .
\end{array}
\]
This expresses that $A$ is an infinite set.
\end{enumerate}
These examples show that the Henkin quantifier \emph{is not expressible in first-order logic}.

\subsection{Further developments}
The next major development was the introduction of IF-logic (``inde\-pendence-friendly logic'') by Jaakko Hintikka and Gabriel Sandu \cite{HS89}.
The intention of IF-logic is to highlight informational dependence and independence. It provides a linear syntax for expressing branching quantification (and more), e.g.  the Henkin quantifier can be written in linear notation as:
\[ \forall x. \, \exists y. \, \forall u. \, (\exists v/x). \, A(x, y, u, v) \]
The ``slashed quantifier''  $(\exists v/x)$  has the intended reading ``there exists a $v$ \emph{not depending on $x$}''. Note the strange syntactic form of this quantifier, with its ``outward-reaching'' scope for $x$.

\paragraph{Dependence Logic}
A simplified approach was introduced by the second author, and developed extensively in the recent monograph \cite{Vaan}.
The main novelty in the formulation of the logic is to use an atomic \emph{dependence predicate\endnote{In \cite{Vaan}
the notation $=\!\!(x_{1}, \ldots , x_{n}, x)$ for $D(x_{1}, \ldots , x_{n}, x)$ is used.}} $D(x_{1}, \ldots , x_{n}, x)$ which holds if \emph{$x$ depends on $x_{1}, \ldots , x_{n}$, and  only on these variables}.
We can then define ``dependence-friendly quantifiers'' as standard quantifiers guarded with the dependence predicate:
\[ (\exists x \setminus x_{1}, \ldots , x_{n}). \, \phi \;\; \equiv \;\; \exists x. (D(x_{1}, \ldots , x_{n}, x) \; \wedge \; \phi) \, . \]
This yields essentially the same expressive power as IF-logic.

\subsection{Compositionality: Hodges' Semantics}

But, what does it all mean?
Hintikka claimed that \emph{a compositional semantics for IF logic could not be given} \cite{Hin98}.
Instead he gave a ``Game-Theoretical Semantics'', essentially reduction to Skolem form as above.

Wilfrid Hodges showed that it could \cite{Hod97a,Hod97b}.\endnote{Hintikka has apparently not conceded the point \cite{Hin02}, although there is no argument as to the mathematical content of Hodges' results. As far as we are concerned, Hodges' semantics meets all the criteria for a compositional semantics, and is moreover fully abstract. Our concern here is to understand it better, as an interesting construction in its own right.}

Before giving Hodges' construction, it will be useful firstly to recall  Tarski's solution to the problem of how to define the truth of a sentence in a first-order structure $\mathcal{M} = (A, \ldots)$ with underlying set $A$.\endnote{The classic reference is \cite{Tar36}, but in fact the modern model-theoretic definition first appeared in \cite{TV56}, as pointed out in Wilfrid Hodges' article on ``Tarski's Truth Definitions'' in the Stanford Encyclopedia of Philosophy, available online at \texttt{http://plato.stanford.edu/entries/tarski-truth/}, which gives an informative overview.} In order to do this, he had to deal with the more general case of open formulas. The idea was to define
\[ \MM, s \models_{X} \phi \]
where $X$ is a finite set of variables including those occurring free  in $\phi$, and $s$ is an assignment of elements of $A$ to $X$.\endnote{Explictly, an assignment is simply a function $s : X \rarr A$. We write $A^{X}$ for the set of all such assignments. Older tradition was to define satisfaction relative to assignments to \emph{all} variables, which were typically arrayed in infinite sequences. More recently, it has been understood, under the influence of categorical logic,  that to reveal the salient structure one should give the definition relative to a finite environment that  grows as quantifiers are stripped off in the recursive definition.}
Typical clauses include:
\[ \begin{array}{lcl}
\MM, s \modX \phi \wedge \psi & \equiv & \MM, s \modX \phi \;\; \mbox{and} \;\; \MM, s \modX \psi \\
\MM, s \modX \neg \phi & \equiv & \MM, s \nmodX \phi \\
\MM, s \modX \forall v. \, \phi & \equiv & \forall a \in A. \; \MM, s[v \mapsto a] \models_{X \cup \{v\}} \phi \\
\MM, s \modX \exists v. \, \phi & \equiv & \exists a \in A. \; \MM, s[v \mapsto a] \models_{X \cup \{v\}} \phi \\
\end{array}
\]
Here $s[v\mapsto a]$ is the assignment defined on $X \cup \{v\}$ as follows: $s[v \mapsto a](v) = a$, and $s[v \mapsto a](w) = s(w)$ for $w \neq v$.

The is the very prototype of a compositional semantic definition.
Via Dana Scott, this idea led to the  use of \emph{environments} in denotational semantics \cite{Sco69}. Environments are nowadays ubiquitous in all forms of semantics in computer science \cite{Win93,Mit96}.

\paragraph{Teams}
Hodges' key idea was to see that one must lift the semantics of formulas from single assignments to \emph{sets of assignments}. Notions of dependence of one variable on others are only meaningful among a set of assignments. Hodges called these sets ``trumps''; we follow \cite{Vaan} in calling them \emph{teams}.

We consider the semantics of Dependence logic \cite{Vaan}. Formulas are built up from standard atomic formulas and their negations and  the dependence predicates, by conjunction, disjunction, and universal and existential quantification. We shall distinguish between the usual atomic formulas (including equality statements) over the first-order signature we are working with, and the dependence formulas. In the case of the standard atomic formulas, we shall also allow their negations, and as usual refer to positive and negated atomic formulas collectively as \emph{literals}. We shall \emph{not} allow negations of dependence formulas; we will see later how to access negative information about dependence, using the new connectives we will introduce in the next section.

The set of all individual variables is denoted $\Var$.
A \emph{team on $X \subseteq \Var$} is a set of Tarski assignments on $X$. We define the following operations on teams:

\begin{itemize}
\item If $T$ is a team on $X$ and $v \in \Var$,  then $T[v \mapsto A]$ is the team on $X \cup \{v\}$ defined by:

\vsn
\[ T[v \mapsto A] = \{ t[v \mapsto a] \mid t \in T \; \wedge \; a \in A \} . \]

\item If $T$ is a team on $X$, $v \in \Var$, and $f : T \lrarr A$, then $T[v \mapsto f]$ is the team on $X \cup \{v\}$ defined by:
\[ T[v \mapsto f] = \{ t[v \mapsto f(t)] \mid t \in T \} . \]
\end{itemize}

\paragraph{The Satisfaction Relation}

We define a satisfaction relation
\[ \MM, T \modX \phi \]
where the free variables of $\phi$ are contained in $X$, and $T$ is a team on $X$. (In practice, we elide $\MM$).

Firstly, for literals $L$  we have:
\[ T \modX L \;\; \equiv \;\; \forall t \in T. \, t \modX L  \]
where $t \modX L$ is the standard Tarskian definition of satisfaction of an atomic formula or its negation  in a structure with respect to an assignment.

\paragraph{Connectives and Quantifiers}
The  clauses for connectives and quantifiers are as follows:

\[ \begin{array}{lcl}
T \modX \phi \wedge \psi & \equiv & T \modX \phi \;\; \mbox{and} \;\; T \modX \psi \\
T \modX \phi \vee \psi & \equiv & \exists \, U, V. \; ([U \modX \phi \;\; \mbox{and} \;\; V \modX \psi] \;\; \wedge \;\; [T = U \cup V]) \\
T \modX \forall v. \, \phi & \equiv & T[v \mapsto A] \models_{X \cup \{v\}} \phi \\
T \modX \exists v. \, \phi & \equiv & \exists f : T \lrarr A. \; T[v \mapsto f] \models_{X \cup \{v\}} \phi .
\end{array}
\]

\paragraph{Semantics of the dependence predicate}
Given a set of variables $X$ and $W \subseteq X$, we define the following notions:

\begin{itemize}
\item An equivalence relation on  assignments on $X$:
\[ s \eqW t \;\; \equiv \;\;  \forall w \in W. \, s(w) = t(w) . \]

\item A function $f : A^{X} \lrarr A$ \emph{depends only on $W$}, written $f : A^{X} \lrW A$, if for some $g : A^{W} \lrarr A$, $f = g \circ p_{XW}$, where $p_{XW} : A^{X} \lrarr A^{W}$ is the evident projection. Note that if such a $g$ exists, it is unique.
\end{itemize}

\noindent Now we can define:
\[ T \modX D(W, v) \;\; \equiv \;\; \forall s, t \in T. \, s \eqW t \; \Rightarrow \; s(v) = t(v) \]
Note that this expresses \emph{functional dependence}, exactly as in database theory \cite{Arm}.

\noindent An equivalent definition can be given in terms of the dependency condition on functions:
\[ T \modX D(W, v) \;\; \equiv \;\; \exists f : T \lrW A. \, \forall t \in T. \, t(v) = f(t) .  \]

\noindent Strictly speaking, this is the ``positive part'' of the definition as given in \cite{Vaan} following Hodges. There is also a negative part, which defines satisfaction for $\phi$ as for the positive definition, but with respect to the De Morgan dual $\phi^{d}$ of $\phi$:
\[ (\phi \vee \psi)^{d} = \phi^{d} \wedge \psi^{d}, \quad (\exists v. \, \phi)^{d} = \forall v. \, \phi^{d}, \quad \mbox{etc.} \]
This allows for a ``game-theoretic negation'', which formally ``interchanges the r\^oles of the players''. It is simpler, and from our perspective loses nothing, to treat this negation as a \emph{defined operation}, and work exclusively  with formulas in negation normal form as above.

\noindent The theory of dependence logic: metalogical properties, connections with second-order logic, complexity and definability issues,  \textit{et cetera}, is extensively developed in \cite{Vaan}. However, as explained in the Introduction, many basic questions remain.
We shall now show how the Hodges semantics can be seen in a new light, as arising from a general construction.

\section{The Hodges construction revisited}
An important clue to the general nature of the construction is contained in the observation by Hodges \cite{Hod97a} (and then in \cite{Vaan}) that the sets of teams denoted by formulas of  IF-logic or Dependence logic are  \emph{downwards closed}: that is, if $T \models  \phi $ and $S \subseteq T$, then $S \models \phi$.
This is immediately suggestive  of well-known constructions on ordered structures.

\subsection{A general construction}
We recall a couple of definitions. A \emph{commutative ordered monoid} is a structure $(M, {+}, 0, {\leqslant})$, where $(M, {\leqslant})$ is a partially ordered set, and $(M, {+}, 0)$ is a commutative monoid (a set with an associative and commutative operation $+$ with unit $0$), such that $+$ is monotone:
\[ x \leqslant x' \; \wedge \; y \leqslant y' \;\; \Rightarrow \;\; x+y \leqslant x' + y' \, . \]
The primary example we have in mind is $\pow{A^{X}}$, the set of all teams on a set of variables $X$, which we think of as the commutative ordered monoid $(\pow{A^{X}}, {\cup}, \vn, {\subseteq})$.

A \emph{commutative quantale} is a commutative ordered monoid where the partial order is a complete lattice, and $+$ distributes over all suprema: $m + \bigvee_{i \in I} m_{i} = \bigvee_{i \in I} (m + m_{i})$.

Let $(M, {+}, 0, {\leqslant})$ be a commutative ordered monoid. Then $\LL(M)$, the set of lower (or downwards-closed) sets of $M$, ordered by inclusion, is the \emph{free commutative quantale} generated by $M$ \cite{MS01}.\endnote{More precisely, it is the left adjoint to the evident forgetful functor.}

A downwards closed subset of a partially ordered set $P$ is a set $S$ such that:
\[ x \leqslant y \in S \; \Rightarrow \; x \in S \, . \]
Thus this notion generalizes the downwards closure condition on sets of teams.
 
\noindent The following notation will be useful.
Given $X \subseteq P$, where $P$ is a partially ordered set, we define
\[ \dset(X) = \{ x \in P \mid \exists y \in X. \, x \leqslant y \} \, , \]
the \emph{downwards closure} of $X$. A set $S$ is downwards closed if and only if $S = \dset(S)$.

As a commutative quantale, $\LL(M)$ is a model of intuitionistic linear logic (phase semantics \cite{Yet,Ros90,Gir87}).\endnote{It is also an instance of Urquhart's semilattice semantics for relevance logic \cite{Urq72}. Mitchell and Simmons observe in \cite{MS01} that in the case (such as ours) where the monoid is a boolean algebra, $\LL(M)$ is actually a model of \emph{classical linear logic}. This does not seem apposite to our purposes here.}
In particular, we have
\[ \begin{array}{rcl}
A \otimes B & = &  {\downarrow}\{ m + n \mid m \in A \; \wedge \; n \in B \} \\
A \linimpl B & = & \{ m \mid \forall n. \, n \in A \Rightarrow m + n \in B \}
\end{array} \]
We note that when the definition of $\otimes$, the multiplicative \emph{conjunction}, is specialized to our concrete setting, it yields the definition of \emph{disjunction} in the Hodges semantics!

The multiplicative implication $\linimpl$ has not been considered previously in the setting of IF-logic and Dependence logic. However, it is perfectly well defined, and is in fact uniquely specified as the adjoint of the linear conjunction:
\[ A \otimes B \leqslant C \;\; \Longleftrightarrow \;\; A \leqslant B \linimpl C \, . \]
Note that linear implication automatically preserves downwards closure.

\subsection{What is the propositional logic of dependence?}
In fact, $\LL(M)$ carries a great deal of structure. Not
 only is it a commutative quantale (and hence carries an interpretation of linear logic), but it is also a \emph{complete Heyting algebra}, and hence carries an interpretation of intuitionistic logic.

We have the clauses
\[ \begin{array}{lcl}
m \models A \wedge B & \equiv & m \models A \; \mbox{and} \; m \models B \\
m \models A \vee B & \equiv & m \models A \; \mbox{or} \; m \models B \\
m \models A \rarr B & \equiv & \forall n \leqslant m. \, \mbox{if} \; n \models A \; \mbox{then} \; n \models B
\end{array}
\]
The situation where we have both intuitionistic logic and multiplicative linear logic coexisting is the setting for \emph{BI logic}, the ``logic of Bunched Implications''  of David Pym and Peter O'Hearn \cite{OHP99,Pym02}, which forms the basis for \emph{Separation logic} (Reynolds and O'Hearn) \cite{Rey02}, an increasingly influential logic for verification.
The construction $\LL(M)$ is exactly the way a ``forcing semantics'' for BI-logic is converted into an algebraic semantics as a ``BI-algebra'', \ie a structure which is both a commutative quantale  \emph{and} a complete Heyting algebra \cite{POHY04}. $\LL(M)$ is in fact the free construction of a complete BI-algebra over an ordered commutative monoid.

This provides one reason for proposing BI-logic as the right answer to the question posed at the beginning of this subsection. The compelling further evidence for this claim will come from the natural
r\^ole played by the novel connectives we are introducing into the logic of dependence. This r\^ole will become apparent in the subsequent developments in this paper.

\subsection{\BID-logic and its team semantics}
We shall spell out the extended logical language we are led to consider, and its concrete team semantics, extending the Hodges-style semantics already given in section~2.

We call the extended language \BID, for want of a better name. Formulas are built from atomic formulas and their negations, and  dependence formulas, by the standard first-order quantifiers, and the following propositional connectives: the intuitionistic (or ``additive'') connectives $\wedge$, $\vee$, $\rarr$, and the multiplicative connectives $\otimes$ and $\linimpl$.

\paragraph{Team Semantics for BI Logic}
The team semantics for \BID-logic is as follows:
\[ \begin{array}{lcl}
T \models A \wedge B & \equiv & T \models A \; \mbox{and} \; T \models B \\
T \models A \vee B & \equiv & T \models A \; \mbox{or} \; T \models B \\
T \models A \rarr B & \equiv & \forall U \subseteq T. \, \mbox{if} \; U \models A \; \mbox{then} \; U \models B \\
T \models A \otimes B & \equiv & \exists U, V. \, T = U \cup V \; \wedge \;  U \models A \; \wedge \; V \models B  \\
T \models A \linimpl B & \equiv & \forall U. \, [U \models A \Rightarrow T \cup U \models  B ]
\end{array}
\]
The clauses for atomic formulas and their negations and for the dependence formulas and quantifiers are as given in section~2.

As already noted, the semantics of $\wedge$ and $\otimes$ coincide with those given for conjunction and disjunction in section~2. The connectives $\vee$ and $\rarr$, intuitionistic or additive disjunction and implication,  and the multiplicative implication $\linimpl$, are new as compared to IF-logic or Dependence logic.

\subsection{The semantics of sentences}
It is worth spelling out the semantics of sentences explicitly. By definition, sentences have no free variables, and there is only one assignment on the empty set of variables, which we can think of as the empty tuple $\langle \rangle$. In the Tarski semantics, there are only two possibilities for the set of satisfying assignments of a sentence, $\vn$ and $\{ \langle \rangle \}$, which we can identify with \emph{false} and \emph{true} respectively. When we pass to the team semantics for \BID-logic, there are three possibilities for down-closed set of teams to be assigned to sentences: $\vn$, $\{ \vn \}$, or $\{ \vn, \{ \langle \rangle \} \}$. Thus the semantics of sentences is \emph{trivalent} in general.

In his papers, Hodges works only with non-empty teams, and has bivalent semantics for sentences. However, there is no real conflict between his semantics and ours. Let $\Bm$ be \BID-logic without the linear implication. Note that $\Bm$ properly contains Dependence logic, which is expressively equivalent to IF-logic \cite{Vaan}.

\begin{proposition}
Every formula in $\Bm$-logic is satisfied by the empty team; hence in particular every sentence of $\Bm$-logic has either $\{ \vn \}$ or $\{ \vn, \{ \langle \rangle \} \}$ as its set of satisfying teams, and the semantics of sentences in $\Bm$-logic is bivalent.
\end{proposition}
\begin{proof}
A straightforward induction on formulas of $\Bm$-logic.
\end{proof}

\noindent On the other hand, linear implication clearly violates this property. Note that the empty team satisfies $A \linimpl B$ if and only if  every team satisfying $A$ also satisfies $B$.
We obtain as an immediate corollary:

\begin{proposition}
Linear implication is not definable in $\Bm$-logic, and \textit{a fortiori}~is not definable in Dependence logic or IF-logic.
\end{proposition}

\subsection{The general Hodges construction}

We shall briefly sketch, for the reader conversant with categorical logic, the general form of the construction.

The standard Tarski semantics of first-order logic is a special case of Lawvere's notion of hyperdoctrine \cite{Law69}. We refer to \cite{Pitts} for a lucid expository account. Construing $\LL$ as a functor in the appropriate fashion, we can give a general form of the Hodges construction as  a functor from classical hyperdoctrines to BI-hyperdoctrines \cite{BBTS05}.
Given a classical hyperdoctrine $\PP : \CCop \lrarr \Pos$, we define a BI-hyperdoctrine  $\HH(\PP)$  on the same base category by composition with the functor $\LL$:
\[ \HH(\PP) = \LL \circ \PP : \CCop \lrarr \Pos \, . \]
Note that $\Pos$ is an order-enriched category, and $\LL$ is an order-enriched functor, so it preserves adjoints, and hence in particular preserves the interpretations of the quantifiers.
This observation is spelled out in more detail in Proposition~\ref{qadjprop}.

This exactly generalizes the concrete Hodges construction, which is obtained by applying $\HH$ to the standard Tarski hyperdoctrine.

A full account will be given elsewhere.

\section{Quantifiers are adjoints in the Hodges construction}
We recall the team semantics for the quantifiers.
\[ \begin{array}{lcl}
T \modX \forall v. \, \phi & \equiv & T[v \mapsto A] \models_{X \cup \{v\}} \phi \\
T \modX \exists v. \, \phi & \equiv & \exists f : T \lrarr A. \; T[v \mapsto f] \models_{X \cup \{v\}} \phi .
\end{array}
\]
We may wonder what underlying principles dictate these definitions.

To answer this question, we firstly recall the fundamental insight due to Lawvere \cite{Law69} that \emph{quantifiers are adjoints to substitution}.\endnote{See \cite{DP02} for an introduction to adjunctions on posets.}

\subsection{Quantifiers as adjoints}
Consider a function $f : X \rarr Y$. This induces a function
\[ f^{-1} : \pow{Y} \lrarr \pow{X} :: T \mapsto \{ x \in X \mid f(x) \in T \} . \]
This function $f^{-1}$ has both a left adjoint $\exists (f) : \pow{X} \lrarr \pow{Y}$, and a right adjoint $\forall (f) : \pow{X} \lrarr \pow{Y}$. These adjoints are uniquely specified by the following conditions. For all $S \subseteq X$, $T \subseteq Y$:
\[ \exists(f)(S) \subseteq T \IFF S \subseteq f^{-1}(T), \;\; \quad  f^{-1}(T) \subseteq S \IFF  T \subseteq \forall(f)(S) . \]
The unique functions satisfying these conditions can be defined explicitly as follows:
\[ \begin{array}{lcl}
\exists(f)(S) &:= & \{ y \in Y \mid \exists x \in X. \, f(x) = y \; \wedge \; x \in S \}\,,  \\
\forall(f)(S) &:= & \{ y \in Y \mid \forall x \in X. \, f(x) = y \; \Rightarrow \; x \in S \}\,.
\end{array}
\]
\noindent Given a formula $\phi$ with free variables in $\{ v_{1}, \ldots , v_{n+1} \}$, it will receive its Tarskian denotation $\lsem \phi \rsem$ in $\Pow(A^{n+1})$ as the set of satisfying assignments:
\[ \lsem \phi \rsem = \{ s \in A^{n+1} \mid s \modX \phi \} \, . \]
We have a projection function
\[ \pi : A^{n+1} \lrarr A^{n} \; :: (a_{1}, \ldots , a_{n+1}) \mapsto (a_{1}, \ldots , a_{n}) \, . \]
Note that this projection is the Tarskian denotation of the tuple of terms $(v_{1}, \ldots , v_{n})$.
We can characterize the standard quantifiers as \emph{adjoints to this projection}:
\[ \lsem \forall v_{n+1}. \, \phi \rsem = \forall(\pi)(\lsem \phi \rsem), \qquad  \lsem \exists v_{n+1}. \, \phi \rsem = \exists(\pi)(\lsem \phi \rsem) \, . \]
If we unpack the adjunction conditions for the universal quantifier, they yield the following bidirectional inference rule:
\[ \infer=[\quad X = \{ v_{1}, \ldots , v_{n} \} \, . ]{\Gamma \vdash_{X} \forall v_{n+1}. \, \phi}{\Gamma \vdash_{X} \phi} \]
Here the set $X$ keeps track of the free variables  in the assumptions $\Gamma$. Note that
the usual ``eigenvariable condition'' is automatically taken care of in this way.

Since adjoints are uniquely determined, this characterization completely captures the meaning of the quantifiers.

\subsection{Quantifiers in the Hodges semantics}

We shall now verify that the definitions of the  quantifiers in the Hodges semantics \emph{are exactly the images under $\LL$ of their standard interpretations in the Tarski semantics}, and hence in particular that they are adjoints to substitution. Thus these definitions are \emph{forced}.

It will be convenient to work with the semantic view of quantifiers, as operators on subsets. Consider formulas with free variables in $\{ v_{1}, \ldots , v_{n+1} \}$. The Tarski semantics over a structure $\mathcal{M} = (A, \ldots)$ assigns such formulas values in $\pow{A^{n+1}}$. We can regard the quantifiers $\exists v_{n+1}$, $\forall v_{n+1}$ as functions
\[ \somepi, \allpi : \pow{A^{n+1}} \lrarr \pow{A^{n}} \]
\[ \begin{array}{lcl}
\somepi(S) & = & \{ s \in A^{n} \mid \exists a \in A. \, s[v_{n+1} \mapsto a] \in S \} \\
\allpi(S) & = & \{ s \in A^{n} \mid \forall a \in A. \, s[v_{n+1} \mapsto a] \in S \}
\end{array}
\]
For any $m$, we define $\HH(A^{m}) = \LL(\pow{A^{m}})$. Thus $\HH(A^{m})$ is the set of downwards closed sets of teams on the variables $\{ v_{1}, \ldots , v_{m} \} $. This provides the corresponding ``space'' of  semantic values for formulas in the Hodges semantics.
The interpretation of quantifiers in that semantics is given by the following set operators:
\[ \someH, \allH : \HH(A^{n+1}) \lrarr \HH(A^{n}) \]
\[ \begin{array}{lcl}
\someH(\UU) & = & \{ T \in \pow{A^{n}} \mid \exists f : T \rarr A. \, T[v_{n+1}  \mapsto f] \in \UU \} \\
\allH(\UU) & = & \{ T \in \pow{A^{n}} \mid T[v_{n+1}  \mapsto A] \in \UU \}
\end{array}
\]
\noindent We extend the definition of $\LL$ to act on functions\endnote{More precisely, homomorphisms of the appropriate kind. The reader familiar with category theory will see that we are really specifying the functorial action of $\LL$ in a particular case.} $h : \pow{Y} \lrarr \pow{X}$:
\[ \LL(h) : \HH(Y) \lrarr \HH(X) :: \UU \mapsto  \dset \{ h(T) \mid T \in \UU \} \, . \]
In the case that $h = f^{-1}$, where $f : X \lrarr Y$, we write $\LL(h) = \HH(f)$.

\begin{proposition}
\label{lqprop}
The Hodges quantifiers are the image under $\LL$ of the Tarski quantifiers:
\[ \someH = \LL(\somepi), \qquad  \allH = \LL(\allpi) \, . \]
\end{proposition}
\begin{proof}
Firstly, we show that $ \LL(\somepi)(\UU) \subseteq \someH(\UU)$ for all $\UU \in \HH(A^{n+1})$.
Suppose that $T \in \UU$. Let $T'  = \somepi(T)$. This means that
\[ \forall t \in T'. \, \exists a \in A. \, t[v_{n+1} \mapsto a] \in T \, . \]
Using the axiom of choice, there exists a function $f : T' \lrarr A$ such that
\[ T'[v_{n+1} \mapsto f] \subseteq T \in \UU \, . \]
Since $\UU$ is downwards closed, this implies that $T' \in \someH(\UU)$, as required.

\noindent The converse follows immediately from the fact that 
\[ \somepi(T[v_{n+1} \mapsto f]) = T \, . \]
Next we show that $ \LL(\allpi)(\UU) \subseteq \allH(\UU)$. Since 
\[ (\allpi(T))[v_{n+1} \mapsto A] \subseteq T \, , \]
if $T \in \UU$, then $\allpi(T) \in \allH(\UU)$ by downwards closure. The converse follows similarly from $T \subseteq \allpi(T[v_{n+1} \mapsto A])$.
\end{proof}

\begin{proposition}
\label{qadjprop}
The Hodges quantifiers are adjoints to substitution:
\begin{enumerate}
\item $\someH$ is left adjoint to $\HH(\pi)$:
\[ \someH(\UU) \subseteq \VVV \;\; \Longleftrightarrow \;\; \UU \subseteq \HH(\pi)(\VVV) \, . \]
\item $\allH$ is right adjoint to $\HH(\pi)$:
\[ \HH(\pi)(\VVV) \subseteq \UU \;\; \Longleftrightarrow \;\; \VVV \subseteq \allH(\UU) \, .\]
\end{enumerate}
\end{proposition}
\begin{proof}
It is straightforward to verify the adjunction conditions directly. We give a more conceptual argument.
There is a natural pointwise ordering on monotone functions between partially ordered sets, $h, k : P \lrarr Q$:
\[ h \leqslant k \;\; \equiv \;\; \forall x \in P. \, h(x) \leqslant k(x) \, . \]
$\LL$ is an \emph{order-enriched functor} with respect to this ordering. Functoriality means that
\[ \LL(h \circ g) = \LL(h) \circ \LL(g), \qquad \LL(\id{M}) = \id{\LL(M)} \, , \]
while order-enrichment means that
\[ h \leqslant k \;\; \Rightarrow \;\; \LL(h) \leqslant \LL(k) \, . \]
These properties imply that $\LL$ automatically preserves adjointness. That is, if we are given monotone maps
\[ f : P \lrarr Q, \qquad g : Q \lrarr P \]
such that $\id{P} \leqslant g \circ f$ and $f \circ g \leqslant \id{Q}$, \ie so that $f$ is left adjoint to $g$, then
\[ \id{\LL(P)} = \LL(\id{P}) \leqslant \LL(g \circ f) = \LL(g) \circ \LL(f) \, , \]
and similarly $\LL(f) \circ \LL(g) \leqslant \id{\LL(Q)}$, so $\LL(f)$ is left adjoint to $\LL(g)$ (and of course $\LL(g)$ is right adjoint to $\LL(f)$). Combining this with Proposition~\ref{lqprop} yields the required result.
\end{proof}

\subsection{The dependence-friendly quantifiers}
We shall also give characterizations of the dependence-guarded quantifiers as certain adjoints: this will be our first use of the intuitionistic implication.

We recall the definition of the dependence-friendly existential quantifier:
\[ (\exists x \setminus x_{1}, \ldots , x_{n}). \, \phi \;\; \equiv \;\; \exists x. (D(x_{1}, \ldots , x_{n}, x) \; \wedge \; \phi) \, . \]
There has not been a comparably natural notion of dependence-friendly universal quantification. According to our analysis, this is because the appropriate connective needed to express the right notion, namely intuitionistic implication, has not been available. Using it, we can define such a quantifier:
\[ (\forall x \setminus x_{1}, \ldots , x_{n}). \, \phi \;\; \equiv \;\; \forall x. (D(x_{1}, \ldots , x_{n}, x) \; \rarr \; \phi) \, . \]
As evidence for the naturalness of these quantifiers, we shall express them both as adjoints.

Firstly, we recall that intuitionistic conjunction and implication are related by another fundamental adjointness \cite{Law69}:
\begin{equation}
\label{impadjeq}
\UU \cap \VVV \subseteq \WW \;\; \Longleftrightarrow \;\; \UU \subseteq \VVV \rarr \WW \, .
\end{equation}
This can be expressed as a bidirectional inference rule:
\[ \infer={\phi \vdash \psi \rarr \theta}{\phi \; \wedge \; \psi \vdash \theta} \, . \]
Next, we extend our semantic notation to the dependence-friendly quantifiers. Given $W \subseteq \{ v_{1}, \ldots , v_{n} \}$, we define $D_{W} \in \HH(A^{n+1})$:
\[ D_{W} = \{ T \mid \forall s, t \in T. \, s \eqW t \; \Rightarrow \; s(v_{n+1}) = t(v_{n+1}) \} \, . \]
Now we can define the semantic operators corresponding to the dependence-friendly quantifiers:
\[ \someW, \allW : \HH(A^{n+1}) \lrarr \HH(A^{n}) \]
\[ \begin{array}{lcl}
\someW(\UU) & = & \someH(D_{W} \cap \UU) \\
\allW(\UU) & = & \allH(D_{W} \rarr \UU)
\end{array}
\]
\begin{proposition}
The dependence-friendly existential $\someW$ is left adjoint to the following operation:
\[ \VVV  \; \mapsto \; (D_{W} \rarr \HH(\pi)(\VVV)) \, . \]
The dependence-friendly universal $\allW$ is right adjoint to the following operation:
\[  \VVV \; \mapsto \; (D_{W} \cap \HH(\pi)(\VVV)) \, . \]
\end{proposition}
\begin{proof}
A direct verification is straightforward, but it suffices to observe that adjoints compose, and then to use Proposition~\ref{qadjprop} and the adjointness~(\ref{impadjeq}).
\end{proof}

Of course, the analysis we have given in this sub-section  applies to any guarded quantifiers; the dependence predicates play no special r\^ole here. The point is to show how the intuitionistic connectives round out the logic in a natural fashion. We shall apply them to a finer analysis of dependence itself in section~\ref{depsec}.

\section{Full Abstraction}
We shall now prove a full abstraction result in the sense of Hodges \cite{Hod97a}.\endnote{As Hodges notes, he himself takes the term, and the concept, from Computer Science \cite{Mil77,Plo77}.}
The point of this is to show that, even if we take sentences and their truth-values as primary,
the information contained in the semantics of formulas in general  is not redundant, since whenever two formulas receive different denotations, they make different contributions overall to the truth-values assigned to sentences.

The fact that such a result holds for $\Bm$-logic is notable, in that the logic is highly non-classical, while the semantics of sentences is bivalent. For \BID-logic, the set of possible truth values for open formulas is huge even in finite models \cite{CH01}, while the semantics of sentences is trivalent.

While our argument follows that of Hodges \cite{Hod97a}, we find a  natural r\^ole for the intuitionistic implication, and can give a very simple proof, while Hodges' argument goes through the correspondence with the game-theoretical semantics.

To formalize  full abstraction, we introduce the notion of a \emph{sentential context} with respect to a set of variables $X$. This is a formula with an occurrence of a ``hole'' $[\cdot]$ such that inserting a formula with free variables in $X$ into the hole yields a sentence.
Now consider two formulas $\phi$ and $\psi$ of \BID-logic, with free variables in $X$. We say that the formulas are \emph{semantically equivalent} if they have the same denotations, \ie the same sets of satisfying teams, in all interpretations with respect to all structures. We say that $\phi$ and $\psi$ are \emph{observationally equivalent} if for all sentential contexts $C[\cdot]$ for $X$, $C[\phi]$ and $C[\psi]$ are assigned  the same truth values in all interpretations. The fact that semantic equivalence implies observational equivalence follows immediately from the compositional form of the semantics. The converse is \emph{full abstraction}.\endnote{While this notion is perfectly consistent with usage in Computer Science, one very important tensioning ingredient in the programming language context is missing, namely correspondence with an independently defined operational semantics \cite{Mil77,Plo77}.}

\begin{proposition}
The team semantics is fully abstract for any sublanguage of \BID-logic containing universal quantification and intuitionistic implication.
\end{proposition}
\begin{proof}
Suppose that $\lsem\phi \rsem \setminus \lsem \psi \rsem$ in some interpretation contains a team $T$. Extend the language with a  relation symbol $R$, and the interpretation by assigning ${\downarrow}(T)$ to $R$. Then use the context
\[ C[\cdot] \; \equiv \; \forall v_{1}, \ldots , \forall v_{n}. \, (R(v_{1}, \ldots , v_{n}) \rarr [\cdot]) \]
where the free variables in $\phi$ and $\psi$ are contained in $\{ v_{1}, \ldots , v_{n} \}$.
Then $C[\phi]$ is true (satisfied by the empty tuple), since for every team $T'$ satisfying $R(v_{1}, \ldots , v_{n})$, $T' \subseteq T$, and hence by assumption and downwards closure, $T'$ satisfies $\phi$. This means that all teams  over  $\{ v_{1}, \ldots , v_{n} \}$ satisfy the implication $R(v_{1}, \ldots , v_{n}) \rarr \phi$, and hence $\langle \rangle $ satisfies $C[\phi]$. On the other hand, $C[\psi]$ is not satisfied by the empty tuple, since $T$ satisfies $R(v_{1}, \ldots , v_{n})$, while $T$ does not satisfy $\psi$ by assumption.
\end{proof}
Note that the use of the intuitionistic implication in relativizing to those teams satisfying the precondition $R(v_{1}, \ldots , v_{n})$ is exactly what is needed.

\section{Analyzing Dependence}
\label{depsec}

We now turn to the dependence predicate itself. Since it encapsulates the ``jump'' from  first-order to second-order semantics, we cannot be too hopeful about taming it axiomatically\endnote{See \cite{MR1867954} for details on this.}. But it turns out that we can give a finer analysis in \BID-logic.

Consider the following ``trivial'' case of dependence:
\[ C(v) \; \equiv \; D(\vn, v) \, . \]
This expresses that $v$ depends on nothing at all, and hence has a fixed value --- functional dependency for the constant function. Semantically, this is the following simple special case of the semantics of dependence:
\[ T \modX C(v) \; \; \equiv \;\; \forall t_{1}, t_{2} \in T. \, t_{1}(v) = t_{2}(v) \, . \]
Using the intuitionistic implication, we can \emph{define} the general dependence predicate from this special case:
\begin{equation}
\label{depeq}
D(W, v) := \left(\bigwedge_{w \in W} C(w) \right) \; \rarr \; C(v) 
\end{equation}

\begin{proposition}
\label{depcprop}
The definition of $D$ from $C$ is semantically equivalent to the definition given previously:
\[ T \modX D(W, v) \;\; \equiv \;\; \forall s, t \in T. \, s \eqW t \; \Rightarrow \; s(v) = t(v) .) \]
\end{proposition}
\begin{proof}
This is just an exercise in unwinding the definitions. Note that the intuitionistic implication lets us range over all subsets of the team which are in a single equivalence class under $\eqW$, and require that $v$ is constant on those subsets.
\end{proof}

\subsection{Armstrong Axioms}
The current stock of plausible axioms for the dependence predicates is limited to the
\emph{Armstrong axioms} from database theory \cite{Arm}. These are a standard complete set of axioms for functional   dependence. They can be given as follows.

\begin{center}
\begin{tabular}{ll}
(1)
 & Always $D(x,x)$. \\
(2) & If $D(x,y,z)$, then $D(y,x,z)$. \\
(3) & If $D(x,x,y)$, then $D(x,y)$. \\
(4) & If $D(x,z)$, then $D(x,y,z)$. \\
(5) & If $D(x,y)$ and $D(y,z)$, then $D(x,z)$. \\
\end{tabular}
\end{center}

\noindent However, in the light of our analysis, the Armstrong axioms simply fall out as standard properties of  implication and conjunction.\endnote{Formal connections between the Armstrong axioms and  propositional logic  were made by Fagin \cite{Fag}. He only considered Horn clauses, so the distinction between intuitionistic and classical logic was not apparent. Nevertheless, the passage to two-element subsets in the ``Semantic proof of the Equivalence Theorem'' in \cite{Fag} implicitly involves similar reasoning to Proposition~\ref{depcprop}.}
If we  set $p = C(x)$, $q = C(y)$, $r = C(z)$, and use (\ref{depeq}) to translate the Armstrong axioms into purely implicational form, we see that they  correspond to the following:

\begin{center}
\begin{tabular}{ll}
(1)
 & $p \rarr p$. \\
(2) & $(p \rarr q \rarr r) \rarr (q \rarr p \rarr r)$. \\
(3) & $(p \rarr p \rarr q) \rarr (p \rarr q)$. \\
(4) & $(p \rarr r) \rarr (p \rarr q \rarr r)$. \\
(5) & $(p \rarr q) \rarr (q \rarr r) \rarr (p \rarr r)$. \\
\end{tabular}
\end{center}
These are the well-known axioms $\textbf{I, C, W, K, B}$ respectively\endnote{Axiom (4) as given generalizes the standard \textbf{K} axiom  $p \rarr q \rarr p$, which is obviously derivable from (1) and (4) by substitution and Modus Ponens.}  \cite{CF58}  --- which form a complete axiomatization of \emph{intuitionistic} (but not classical!) implication.\endnote{Under the (Curry part of the) Curry-Howard correspondence, they correspond to a well-known functionally complete set of \emph{combinators} \cite{CF58}.}
A standard example of a classically valid implicational formula which is \emph{not} derivable from these axioms is \emph{Peirce's law}: $((p \rarr q) \rarr p) \rarr p$.

Thus we have reduced the understanding of the dependence predicate to understanding of the, \textit{prima facie}~simpler, constancy predicate $C$.

\section{Further Directions}
In this final section, we shall sketch a number of further directions.
Detailed accounts are under development, and  will appear elsewhere.

\subsection{Completeness}
Predicate BI-logic is a well developed formalism, with a proof theory which is sound and complete relative to an algebraic semantics \cite{Pym02}. Since \BID-logic is a special case, we have a sound ambient inference system. Of course this is not complete for the intended semantics for \BID-logic  --- and cannot be.
We may hope to obtain completeness for some smaller class of models, possibly on the lines of the Henkin completeness theorem for higher-order logic \cite{Hen50}.

\subsection{Diagrams}
Now fix a particular interpretation in a structure $\MM$  with universe $A$. Consider the following construction. We  introduce constants for each $a \in A$, the usual first order diagram (all true atomic sentences), and the following  infinitary axiom:
\[ \forall v. \, \bigotimes_{a \in A} (v=a) \, . \]
We can \emph{define} the predicate $C$ (and hence dependence $D$) by the following infinitary formula:
\[ C(v) := \bigvee_{a \in A} (v=a) \, . \]
Note how the two different connectives (one additive, the other multiplicative) feature naturally.

This gives a logical (albeit infinitary) characterization of dependence.

\subsection{Representation}
We can also consider representation theory for the structures $\HH(X) = \LL(\Pow(X))$. We seek lattice-theoretic properties of these structures which suffice to characterize them.

Firstly, we note that the down-closures of single teams are exactly the \emph{complete join-primes} of the lattice:
\[ a \leqslant \bigvee_{i} b_{i} \;\; \Rightarrow \;\; \exists i. \, a \leqslant b_{i} . \]
Moreover, these join-primes order generate, \ie every element is the join of the join-primes below it.
All of this structure is in terms of the intuitionistic disjunction.

Next, we note that the join-primes are closed under $\otimes$, which is  moreover idempotent on the join-primes, endowing them with the structure of a semilattice.
This is very different to the semilattice structure given by intuitionistic disjunction: e.g.
\[ \dset(T_{1}) \vee \dset(T_{2}) = \dset(\{T_{1}, T_{2}\}) \neq  \dset(T_{1} \cup T_{2}) = \dset(T_{1}) \otimes \dset(T_{2}) \, . \]
The \emph{double singletons} are exactly the complete atoms in this semilattice, which is complete atomic in the usual sense.

Syntactically, assuming names for elements, we can describe these atomic join-primes in the lattice of propositions over variables $v_{1}, \ldots , v_{n}$ as
\[  (v_{1} = a_{1}) \; \wedge \; \cdots \; \wedge \;  (v_{n} = a_{n}). \]
These are of course the tuples. (Downclosures of) arbitrary teams are then described by expressions
$\bigotimes_i A_i$,  where $A_i$ ranges over such atoms. Arbitrary elements are joins (intuitionistic disjunctions) of such elements. So there is a normal form for general elements:
\[ \bigvee_{i} \bigotimes_{ij} A_{ij} \, . \]
Moreover, from the lattice-theoretic properties it is easily shown that the ordering between such normal forms agrees with the set inclusion ordering.

\subsection{Expressiveness}
\newcommand{\rel}[1]{\mbox{rel}(#1)}
One of the defining characteristics of Dependence Logic as well as IF-logic is that they can be expressed in Existential Second Order Logic, $\Sigma^1_1$, and conversely, every $\Sigma^1_1$ definable property of structures can be expressed with a sentence of Dependence Logic. Both are true even on finite structures. To see what this connection with $\Sigma^1_1$ means let us adopt the notation that if $T$ is a team on a set $X$ of variables, then $\rel{T}$ is the corresponding relation.
Hodges~\cite{Hod97b} associates with every formula $\phi$ of IF-logic (equivalently, of Dependence Logic) with free variables in the set $X=\{x_1,...,x_n\}$ an Existential Second Order sentence $\tau_\phi(R)$, with $R$ an $n$-ary predicate symbol, such that in any model $\MM$ and for any team $T$ on $X$ the following holds: \begin{equation}\label{sol}\MM,T\models_X\phi\iff (\MM,\rel{T})\models\tau_\phi(R).\end{equation} Conversely, if $\Phi$ is any Existential Second Order sentence, then there is a sentence $\phi$ of Dependence Logic such that the following holds for all models $\MM$: $$\MM\models\Phi\iff\MM,\{\langle\rangle\}\models\phi.$$
Virtually all model theoretic properties of Dependence Logic follow from this relationship with $\Sigma^1_1$, for example, the Compactness Theorem, the downward and upward L\"owenheim-Skolem Theorems, the Interpolation Theorem, and the fact that every sentence $\phi$ in Dependence Logic for which there exists a ``negation" $\psi$ such that for all $\MM$ $$\MM\models\phi\iff\MM\not\models\psi,$$ is actually first order definable\endnote{See e.g. \cite{Vaan} for details.}. Also the interesting fact that the class of properties of finite structures expressible in Dependence Logic is exactly NP follows from this. Because of these connections it is quite interesting to ask whether the extensions $\Bm$ and $\BID$ can likewise be embedded in $\Sigma^1_1$, the existential fragment of Second Order Logic.

Now the question arises which semantics one should use. To be able to compare results with Dependence Logic and IF-logic, we use the full semantics familiar from \cite{Hod97a} and \cite{Vaan}.

\begin{proposition}
There is no translation of  any extension of  Dependence Logic containing either intuitionistic implication or linear implication into existential second order $\Sigma^1_1$. The same is true on finite models, assuming NP$\ne$co-NP.
\end{proposition}

\begin{proof} Let $\phi(x_1)$ be a formula of Dependence Logic in the empty vocabulary such that for any team $T$: $\MM,T\models\phi(x_1)$ if and only if $A$ is infinite\endnote{The free variable $x_1$ plays no role in this.}. Let $\bot$ denote a sentence in the empty vocabulary, only satisfied by the empty team, e.g. $\forall x.x=x\wedge\neg x=x$.
Suppose there were an Existential Second Order sentence $\tau(R)$ such that
a model $\MM$ and a team $T$ on $\{x_1\}$ satisfy $\phi(x_1)\to\bot$ if and only if $(\MM,\rel{T})$ satisfies $\tau(R)$. If $\MM$ is any finite model and $T=\{s\}$, where $s(x_1)\in M$, then  $\MM,T\models\phi(x_1)\to\bot$, whence $(\MM,\{s(x_1)\})\models\tau(R)\wedge\exists x_1R(x_1)$. By the Compactness Theorem of Existential Second Order Logic, $\tau(R)\wedge\exists x_1R(x_1)$ has an infinite model $(\MM',\rel{T'})$. Thus $\MM'$ and the team $T'$ satisfy $\phi(x_1)\to\bot$. Moreover, $T'\ne\vn$. By the definition of the semantics of $\to$, since $T'$ satisfies $\phi(x_1)$ in $\MM'$, $T'$ must satisfy $\bot$, a contradiction.

Let us then consider finite models. It is easy to write down   a formula $\phi(x_1)$ of Dependence Logic in the vocabulary of graphs such that for any team $T\ne\vn$: $\MM,T\models\phi(x_1)$ if and only if $\MM$ is 3-colorable. Let $\bot$ be as above. If $\MM$ is any graph that is not 3-colorable and $T=\{s\}$, where $s(x_1)\in M$, then  $\MM,T\models\phi(x_1)\to\bot$. On the other hand, suppose $\MM$ is 3-colorable, but $\MM$ and some team $\{s\}$ satisfy $\phi(x_1)\to\bot$.  By the definition of the semantics of $\to$, since $\{s\}$ satisfies $\phi(x_1)$ in $\MM$, $\{s\}$ must satisfy $\bot$, a contradiction. Thus a graph $\MM$ and a team $\{s\}$ satisfy $\phi(x_1)\to\bot$ if and only if $\MM$ is not 3-colorable. Suppose now there were an Existential Second Order sentence $\tau(R)$ such that
a graph $\MM$ and a team $T$ satisfy $\phi(x_1)\to\bot$ if and only if $(\MM,\rel{T})$ satisfies $\tau(R)$. Then we could check if a graph $\MM$ is not 3-colorable by checking if  $\tau(R)$ is satisfied by $\MM$ and and a team $\{s\}$, where $s$ can be any assignment. The latter is NP, so we get NP=co-NP.

The same argument can be used to show that $\linimpl$ leads outside of $\Sigma^1_1$: Suppose $\phi(x_1)$ is as above and there is an Existential Second Order sentence $\tau(R)$ such that
a model $\MM$ and a team $T$ satisfy $\bot \, \wedge \, (\phi(x_1)\linimpl\bot)$ if and only if $(\MM,\rel{T})$ satisfies $\tau(R)$. If $\MM$ is any finite model and $T=\vn$, then  $\MM,T\models\bot \, \wedge \, (\phi\linimpl\bot)$, whence $(\MM,\vn)\models\tau(R)$. By the Compactness Theorem of Existential Second Order Logic, $\tau(R)$ has an infinite model $(\MM',\rel{T'})$. Thus $\MM'$ and the team $T'$ satisfy $\bot \, \wedge \, (\phi\linimpl\bot)$. In particular, $T'=\vn$ and $\vn$ satisfies $\phi\linimpl\bot$. Since in this model any $\{s\}$ satisfies $\phi$, by the definition of the semantics of $\linimpl$, $\{s\}$ satisfies $\bot$, a contradiction.
\end{proof}

\noindent The proof actually shows that $\BID$ fails to satisfy the Compactness Theorem. A similar argument shows that $\BID$ fails to satisfy the Downward L\"owenheim Skolem Theorem. 

\begin{proposition}
There is a translation of $\BID$ into Full Second Order Logic.
\end{proposition}

\begin{proof} We follow \cite{Hod97b} (see also \cite{Vaan}) and present only the additions needed over and above Dependence Logic and IF-logic:
\begin{eqnarray*}
\tau_{\phi\linimpl\psi}(R) &=&\forall S(\tau_\phi(S)\to\forall U(\forall \vec{x}(U(\vec{x})\leftrightarrow(S(\vec{x})\vee
R(\vec{x})))\to\tau_\psi(U)))\\
\tau_{\phi\to\psi}(R) &=&\forall S(\forall \vec{x}(S(\vec{x})\to
R(\vec{x}))\to(\tau_\phi(S)\to\tau_\psi(S))) \, .
\end{eqnarray*}\end{proof}

\noindent In conclusion, we may say that $\Bm$ and $\BID$ seem to have a more robust and uniform algebraic structure than Dependence Logic and IF-logic. We anticipate that this is reflected also in an effective proof theory, still to be developed. On the other hand the price of this seems to be that ``nice" model theoretic properties are lost, at least in the full semantics. Perhaps there are some underlying, hitherto unidentified, reasons why logics developed for dependence cannot simultaneously have a ``nice" model theory and effective proof theory. After all, we know from Lindstr\"om's Theorem (\cite{MR0244013}) that there are intrinsic obstacles to having model-theoretically defined extensions of first order logic with both nice proof theory and nice model theory. However, we have a trivalent logic, unlike the setting considered by Lindstr\"om. So it is too early to say whether there are general reasons why $\BID$ does not satisfy Compactness and other model theoretic properties familiar from Dependence Logic, or whether we  have just not found the right concepts yet.

\theendnotes

\bibliographystyle{klunamed}

\bibliography{bibfile}

\end{article}
\end{document}